\documentclass[11pt]{amsart}
\usepackage{
amssymb,
amsmath}

\usepackage[bookmarks,colorlinks,pagebackref]{hyperref}

\usepackage{hyperref}

\usepackage[latin1]{inputenc}

\usepackage{mathpazo}
\usepackage[scaled=.95]{helvet}
\usepackage{courier}

\usepackage[all,cmtip,2cell]{xy}

\numberwithin{equation}{section}
\theoremstyle{plain} 
\newtheorem{proposition}{Proposition}[section]  

\newtheorem{theorem}[proposition]{Theorem} 

\theoremstyle{definition} 

\newtheorem{example}[proposition]{Example} 

\newtheorem{question}[proposition]{Question}

\newtheorem{construction}[proposition]{Construction}
\newtheorem{notation and recalls}[proposition]{Notations and Recalls}

\newcommand\Supp{\operatorname{Supp}}

\newcommand\Ann{\operatorname{Ann}}

\newcommand\Hom{\operatorname{Hom}}

\newcommand\Ext{\operatorname{Ext}}
\newcommand\Rad{\operatorname{Rad}}

\newcommand\Coker{\operatorname{Coker}}
\newcommand\im{\operatorname{Im}}

\newcommand{\xx}{\underline x}

\newcommand{\TT}{\underline T}
\newcommand{\au}{\underline a}
\newcommand\grade{\operatorname{grade}}

\newcommand\cd{\operatorname{cd}}

\newcommand\height{\operatorname{height}}
\newcommand\Spec{\operatorname{Spec}}

\newcommand{\qism}{\stackrel{\sim}{\longrightarrow}}

\author[P.~Schenzel]{Peter Schenzel}
\title["Infinite" properties]
{"Infinite" properties of certain local cohomology modules of determinantal rings}
	
\address{Martin-Luther-Universit\"at Halle-Wittenberg,
Institut f\"ur Informatik, D --- 06 099 Halle (Saale), Germany}
\email{peter.schenzel@informatik.uni-halle.de}

\date{\today}
\dedicatory{Dedicated to Nguyen Tu Cuong on the occasion of his 70th birthday \\and for a 
		long friendship}

\begin{document}

\begin{abstract} 
For given integers $m,n \geq 2$ there are examples of ideals $I$ of complete 
determinantal local rings $(R,\mathfrak{m}), 
\dim R = m+n-1, \grade I = n-1,$ with the canonical module $\omega_R$ and the property 
that the socle dimensions of $H^{m+n-2}_I(\omega_R)$ and 
$H^m_{\mathfrak{m}}(H^{n-1}_I(\omega_R))$ are not 
finite. In the case of $m = n$, i.e. a Gorenstein ring, the socle dimensions 
provide further information about the $\tau$-numbers as studied in \cite{MS}.  
Moreover, the endomorphism ring of $H^{n-1}_I(\omega_R)$ is studied and shown 
to be an $R$-algebra of finite type but not finitely generated as $R$-module generalizing 
an example of \cite{Sp6}. 
\end{abstract}


\maketitle


\section{Introduction}
Let $I$ denote an ideal of a local ring $(R,\mathfrak{m})$ with 
$\Bbbk = R/\mathfrak{m}$ its residue field. Let $M$ 
be a finitely generated $R$-module, and let $H^i_I(M), i \in \mathbb{Z},$ 
denote the local cohomology modules of $M$ with respect to $I$ (see 
\cite{Ga2} or \cite{BrS} for definitions). By $\grade (I,M)$ we denote the length 
of the largest regular sequence contained in $I$. In the case of $M= R$ ,
an equicharacteristic complete regular local ring,  the following is known:
\begin{itemize}
	\item[(a)] The Bass numbers $\dim_{\Bbbk} \Ext_R^i(\Bbbk,H^i_I(R))$ are finite 
	(see \cite{HS} and \cite{Lg1}). 
	\item[(b)] The natural homomorphism $R \to \Hom_R(H^c_I(R),H^c_I(R)), c = \grade I,$
	is an isomorphism if and only if $H^i_I(R) = 0$ for $i = d-1,d,$ and $c < d-1, d = \dim R$ 
	(see \cite{Sp6}).
\end{itemize}
The socle $\Hom_R(\Bbbk, H^i_I(M))$ is in general not finite dimensional as has been 
shown at first by R. Hartshorne (see \cite{Hr4}) by disproving a question 
by A. Grothendieck about cofiniteness of local cohomology, i.e. the finiteness of $\Hom_R(R/I, H^i_I(R))$. 
In their paper Huneke and Koh 
(see \cite{HK}) studied cofiniteness of various ideals. Moreover they showed  that for a field $\Bbbk$ of 
characteristic zero  and a polynomial ring $R = \Bbbk[x_{i,j}, 1 \leq 2, 1 \leq j \leq 3]$ the local cohomology 
$H^3_I(R)$ is not $I$-cofinite for $I$ the ideal generated by the maximal minors of the 
$2 \times 3$-matrix $(x_{i,j})$. A large class of 
local cohomology modules with infinite socle has been constructed by 
T. Marley and C. Vassilev (see \cite{MV}). 

The main topic of the present note is to discuss and relate the properties (a) and (b)
above in the non-smooth situation of certain determinantal rings. More precisely:
\begin{theorem} \label{rem-1}
	Let $R$ denote the completion of the coordinate ring of the vanishing ideal of the 
	$2\times2$ minors of an $m\times n$-matrix of variables and therefore $\dim R = m+n-1$. For the ideal $I$ generated by the first $n-1$ columns we have $\dim R/I = m$ and 
	$\grade I = n-1$. Let $\omega_R$ denote the canonical module of $R$.
	\begin{itemize}
		\item[(a)] $\Hom_R(\Bbbk, H^{m+n-2}_I(\omega_R))$ and 
		$\Hom_R(\Bbbk, H^{m+n-2}_I(R))$ are not finite dimensional, that is $I$ 
		is not cofinite and $H^i_I(\omega_R) = 0$ for $i \not= n-1,m+n-2$.
		\item[(b)] $\Hom_R(H^{n-1}_I(\omega_R),H^{n-1}_I(\omega_R))$ is an $R$-algebra of finite type, not finitely generated over $R$.
		\item[(c)] $\Hom_R(\Bbbk,H^m_{\mathfrak{m}}(H^{n-1}_I(\omega_R)))$ 
		and $\Hom_R(\Bbbk,H^0_{\mathfrak{m}}(H^{m+n-2}_I(\omega_R)))$ are 
		not finite dimensional.
	\end{itemize}
\end{theorem} 
  
In the case of $m=n$ the ring $R$ is a Gorenstein ring and $\omega_R \cong R$. For $m=n=2$ we  recover R. Harsthorne's example (see \cite{Hr4}) by completely different arguments. For an ideal $I$ the cohomological dimension is defined by $\cd I = \sup \{i \in \mathbb{N}| H^i_I(R) \not= 0\}$ (introduced by R. Hartshorne (see \cite{Hr2})).
Whence, $\cd I = m+n-2$ for the examples above. That is, possible non-cofinite 
ideals could occur in the highest non-vanishing cohomological level. In the case of a $d$-dimensional 
Gorenstein ring $(R,\mathfrak{m})$ the socle dimension $\Hom_R(\Bbbk, H^i_{\mathfrak{m}}(H^{d-j}_I(R)))$ 
are called $\tau$-numbers of type $(i,j)$ of $I$ (see \cite{MS}). In the case of a regular local ring containing 
a field they coincide with the Lyubeznik numbers introduced by Lyubeznik in \cite{Lg1} (see \cite[3.5]{MS} for the 
details). Therefore, the above results yield 
further information about these $\tau$-numbers. 

In Section 2 we prove the essentials about local cohomology for our purposes. Theorem 
\ref{thm-1} is the needed technical result for our constructions. Furthermore, we 
introduce the basics for the construction of our examples, based about some results 
of Segre varieties. In Section 3 we prove the statements of the examples 
introduced in Section 2. For the needed results  about Commutative Algebra we refer to Matsumura's 
book \cite{Mh}. For a few homological arguments resp. some facts about determinantal ideals we refer to \cite{SS} resp. to \cite{BrH}.

\section{Preliminaries and constructions}
At the beginning let us recall a few basics from local duality. 

\begin{notation and recalls} \label{not-1}
	(A) Let $(R,\mathfrak{m})$ denote a $d$-dimensional Cohen-Macaulay ring 
	with $E = E_R(R/\mathfrak{m})$ the injective hull of its residue field. Suppose that 
	$R$ possesses a normalized dualizing complex $D$ in the sense of 
	\cite[11.4.6]{SS}. Then $R$ admits a canonical module $\omega_R$ (see 
	\cite{Sp1}). Moreover, for an arbitrary $R$-module $X$ there is the following 
	Local Duality Theorem 
	\[
	H^i_{\mathfrak{m}}(X) \cong \Hom_R(\Ext_R^{d-i}(X,\omega_R),E)
	\]
	for all $i \geq 0$ (see e.g. \cite[10.4.3]{SS}).  For further properties of $\omega_R$ 
	we refer to \cite{SS} and \cite{Sp1}. \\
	(B) With the notation of (A) let $M$ denote a finitely generated $R$-module. Then 
	$\omega_R(M)$, the canonical module of $M$ (in the sense of \cite{Sp1}), exists 
	and there is an isomorphism $\Ext_R^c(M,\omega_R) \cong \omega_R(M)$, where 
	$c = \dim R - \dim_RM$ (see also 
	\cite{Sp1} for more details). \\
	(C) With the notation of (A) it follows that $\omega_R \qism D$, a minimal injective 
	resolution of $\omega_R$, is a normalized dualizing complex in the sense of 
	\cite[11.4.6]{SS}. Let $M$ denote a finitely generated $R$-module. Then there 
	is a natural morphism 
	\[
	M \to \Hom_R(\Hom_R(M,D),D)
	\]
	that is an isomorphism in cohomology. Let $c = d -\dim_RM$, then it induces a 
	natural homomorphism 
	\[
	M \to \Ext_R^c(\Ext_R^c(M,\omega_R),\omega_R) = \omega_R(\omega_R(M)).
	\]
	(D) It is known that for an $R$-module $M$ the module $\omega_R(\omega_R(M))$ is the 
	$S_2$-fication of $M$ (see \cite{Sp1}). Moreover, the natural homomorphism 
	$M \to \omega_R(\omega_R(M))$ induces a short exact sequence 
	\[
	0 \to M/u(M) \to \omega_R(\omega_R(M)) \to C \to 0,
	\]
	where $u(M)$ denotes the intersection of those primary components of $M$ that are 
	of highest dimension $\dim_RM$ and $\dim_R C \leq \dim_R M -2$ for $C$ the cokernel.
	Let $I \subset R$ an ideal of grade $c$ in the Cohen-Macaulay ring $R$. Then we get 
	an inverse system of homomorphisms $R/I^{\alpha} \to \Ext_R^c(\Ext_R^c(R/I^{\alpha},\omega_R), \omega_R)$.  By passing to the inverse limit 
	there is a homomorphism 
	\[
	\hat{R}^I \to B := \varprojlim 
	\Ext_R^c(\Ext_R^c(R/I^{\alpha},\omega_R), \omega_R) \cong \Ext_R^c(H^c_I(\omega_R),\omega_R),
	\]
	where $\hat{R}^I$ denotes the 
	$I$-adic completion of $I$. Note that the inverse maps are induced by 
	$R/I^{\alpha +1} \to R/I^{\alpha}$. By a slight modification  
	of the arguments of \cite[Section 3]{Sp6} it follows that $B$ is a commutative ring. \\
	(E) For an arbitrary $R$-module $X$ there is a natural homomorphism $R \to 
	\Hom_R(X,X), r \mapsto \phi(r)$, the multiplication by $r \in R$ on $X$. Its kernel is 
	$\Ann_RX$. In general the endomorphism ring $\Hom_R(X,X)$ is not commutative. 
\end{notation and recalls}

For the property of the ring $B$ in \ref{not-1} (D) we need a few more 
intrinsic properties.

\begin{theorem} \label{thm-1}
	Let $(R,\mathfrak{m},\Bbbk)$ be a $d$-dimensional complete local Cohen-Macaulay 
	ring. Let $\omega_R$ denote its canonical module. For an ideal $I \subset R$ with 
	$c = \grade I$ we have:
		\begin{itemize}
			\item[(a)] There are isomorphisms 
			\[
			\Hom_R(H^c_I(\omega_R),H^c_I(\omega_R)) \cong \Ext_R^c(H^c_I(\omega_R),\omega_R) \cong \Hom_R(H^{d-c}_{\mathfrak{m}}(H^c_I(\omega_R)),E).
			\]
			\item[(b)] The endomorphism ring $\Hom_R(H^c_I(\omega_R),H^c_I(\omega_R))$ 
			is a finitely generated $R$-module if and only if $\dim_{\Bbbk} \Hom_R(\Bbbk,H^{d-c}_{\mathfrak{m}}(H^c_I(\omega_R)))$ is finite.
			\item[(c)] The natural ring homomorphism $\rho :R \to  \Hom_R(H^c_I(\omega_R),H^c_I(\omega_R))$ is onto if and only if 
			$\dim_{\Bbbk} \Hom_R(\Bbbk,H^{d-c}_{\mathfrak{m}}(H^c_I(\omega_R))) = 1$
		\end{itemize}
\end{theorem}

\begin{proof}
	(a): In order to prove the first isomorphism we refer to \cite[2.2 (d)]{Sp8}. In order to proof the 
	second one we have that 
	\[
	\Ext_R^c(H^c_I(\omega_R),\omega_R) \cong \Hom_R(H^{d-c}_{\mathfrak{m}}(H^c_I(\omega_R)),E)
	\]
	as follows by the Local Duality for the complete local Cohen-Macaulay ring $R$ 
	(see \ref{not-1} (A)). \\
	(b): By virtue of (a) we have isomorphisms 
	$B/\mathfrak{m}^{\alpha} \cong \Hom_R(\Hom_R(R/\mathfrak{m}^{\alpha},H^{d-c}_{\mathfrak{m}}(H^c_I(\omega_R)),E)$
	that form an inverse system. By passing to the inverse limit there are isomorphisms 
	\[
	\hat{B}^{\mathfrak{m}} = \varprojlim \Hom_R(\Hom_R(R/\mathfrak{m}^{\alpha},H^{d-c}_{\mathfrak{m}}(H^c_I(\omega_R)),E) 
	\cong \Hom_R(H^0_{\mathfrak{m}}(H^{d-c}_{\mathfrak{m}}(H^c_I(\omega_R))),E) \cong B
	\]
	because of $\Supp_R H^{d-c}_{\mathfrak{m}}(H^c_I(\omega_R)) \subseteq V(\mathfrak{m})$. 
	That is, $B$ is $\mathfrak{m}$-adic complete. Then $B$ is a finitely generated 
	$R$-module if and only if $\dim_{\Bbbk} B/\mathfrak{m}B < \infty$ (see e.g. \cite[Theorem 8.4]{Mh}). \\
	(c): Since $H^c_I(\omega_R) \not= 0$ 
	it follows that the natural homomorphism 
	\[
	\rho \otimes 1_{\Bbbk}: R\otimes_R \Bbbk \to 	\Hom_R(H^c_I(\omega_R),H^c_I(\omega_R)) \otimes_R\Bbbk
	\]
	is not zero. Then the statement follows by (b).  
\end{proof}

\begin{construction} \label{con-1}
	(A) Let $\Bbbk$ denote a field. Let $m, n \geq 2$ denote integers and let 
	$$
	\mathcal{X} = \begin{pmatrix} x_{11} & \ldots & x_{1n} \\
	\vdots & \ddots & \vdots\\ x_{m1} & \ldots & x_{mn} \end{pmatrix}
	$$ 
	denote a $m \times n$ matrix of 
	$mn$ variables over $\Bbbk$. Let $P = \Bbbk[[\mathcal{X}]]$ 
	denote the formal powers series ring in the $mn$ variables of $\mathcal{X}$. 
	We define $R = P/I_2$, where 
	$I_2 = I_2(\mathcal{X})$ denotes the ideal generated by the $2 \times 2$ 
	minors of $\mathcal{X}$. Then $R$ is a local Cohen-Macaulay domain of dimension $m+n-1$ 
	(see \cite{HR}).\\
	(B) We put $\xx_i = x_{1i}, \ldots, x_{mi}, i = 1,\ldots,n,$ the elements of the 
	$i$-th column of $\mathcal{X}$ and $\xx = \xx_n$ the elements of the last column. 
	We define $I = (\xx_1, \ldots,\xx_{n-1})R$ the ideal generated by the elements 
	of the first $n-1$ columns of $\mathcal{X}$. Then $\dim R/I = m$ and 
	$\height I = \grade I = n-1$.
	In the following we recall the form ring $G_I(R)$ of $I \subset R$. First we define 
	a $m \times n$-matrix 
	\[
	\mathcal{Y} = \begin{pmatrix} T_{11} & \ldots & T_{1,n-1}& x_{1n} \\
	\vdots & \ddots & \vdots & \vdots \\ T_{m1} & \ldots & T_{m,n-1} & x_{mn}\end{pmatrix}
	\]
	where $T_{ij}, 1 \leq i \leq m, 1 \leq j \leq n-1,$ are variables over $R$. 
	Let $J = I_2(\mathcal{Y})$ the ideal generated by the $2 \times 2$-minors of $\mathcal{Y}$.
	Let $\mathcal{T}$ denote the set of all variables $T_{i,j}, 1 \leq i \leq m, 1 \leq j \leq n-1$ of degree 1 and let $\mathcal{G} =(R/I)[\mathcal{T}]$. Then consider the 
	induced homogeneous map
	\[
	\Theta: \mathcal{G}/J \to Gr_I(R), \; T_{ij} +J \mapsto x_{ij} +I/I^2, \; 1 \leq i \leq m, 1 \leq j \leq n-1.
	\] 
	Note that it is well defined and surjective. Clearly, $\mathcal{G}/J$ 
	is a Cohen-Macaulay determinantal domain with $\dim \mathcal{G}/J = m+n-1$, i.e. 
	$J$ is a prime ideal. Moreover 
	$Gr_I(R)$ is a ring with $\dim Gr_I(R) = m+n-1$. Let $\mathcal{J}$ be the preimage of $\ker \Theta$ in $\mathcal{G}$. 
	Let $\mathcal{P} \subset \Spec \mathcal{G}$ denote the preimage of a highest dimensional prime ideal in $Gr_I(R)$. 
	Then $J \subseteq \mathcal{J} \subseteq \mathcal{P}$ and $\dim \mathcal{G}/J = 
	\dim \mathcal{G}/\mathcal{P}$ and therefore $J = \ker \Theta$. That is, $\Theta$ is an isomorphism and $Gr_I(R)$ is a domain. 
	\\
	(C) Next we introduce the ring $S = \Bbbk[[\xx,\au]]$, where $\xx = x_1,\ldots,x_m$ with $x_i= x_{in}, i = 1,\ldots,m$ and $\au = a_1,\ldots,a_{n-1}$ are sequences of 
	variables over $\Bbbk$.	Then $S$ is a regular local ring with $\au = a_1,\ldots,a_{n-1}$ a regular sequence. Therefore $Gr_{(\au)}(S) \cong (S/\au)[\TT]
	\cong \Bbbk[[\xx]][\TT]$ with 
	$\TT = T_1,\ldots,T_{n-1}$ variables of degree $1$. We define a ring homomorphism 
	\[
	\phi : R \to S, \; x_{ij} \mapsto a_i x_j, \; i = 1,\ldots, n-1, j = 1,\ldots,m.
	\]
	Note that $S$ is a domain and $\dim R = \dim S = m+n+1$. Therefore 
	$\phi$ is injective. The homomorphism $\phi$ extends to a homomorphism 
	\[
	\psi : Gr_I(R) \to Gr_{(\au)}(S), \; T_{ij} +J \mapsto x_jT_i, \; i = 1,\ldots, m, j = 1,\ldots,n-1.
	\]
	Since both $Gr_I(R)$ and $Gr_{(\au)}(S)$ are domains of the same dimension $\psi$ 
	is injective. 
\end{construction}

\section{Examples and remarks}
In the following we denote by $\omega_R$ the canonical module of a Cohen-Macaulay 
ring $R$. For basic definitions and properties of $\omega_R$ we refer to \cite[3.3]{BrH}, 
 the generalization in \cite{Sp1} and the summary in \ref{not-1}.  For a local ring $(R,\mathfrak{m})$ 
we denote by $E = E_R(R/\mathfrak{m})$ the injective hull of the residue field. 

\begin{example} \label{ex-1}
	We fix the notation of \ref{con-1}, that is,  $R$ 
	is the determinantal Cohen-Macaulay ring of dimension $m+n-1$ 
	with $I = (\xx_1, \ldots,\xx_{n-1})R$ the ideal generated by the elements 
	of the first $n-1$ columns of $\mathcal{X}$.  
	Then there are the following results:
	\begin{itemize}
		\item[(a)] $H^i_I(\omega_R) \not= 0, i = n-1, m+n-2,$ that is, $\cd(I) = m+n-2$.
		\item[(b)] $\dim_{\Bbbk} \Hom_R(\Bbbk, H^{m+n-2}_I(\omega_R))$ and 
		$\dim_{\Bbbk} \Hom_R(\Bbbk, H^{m+n-2}_I(R))$ are not finite, that is $I$ 
		is not cofinite.
		\item[(c)] $\Hom_R(H^{n-1}_I(\omega_R),H^{n-1}_I(\omega_R))$ is a commutative 
		Noetherian ring, not finitely generated over $R$.
		\item[(d)] $\dim_{\Bbbk} \Hom_R(\Bbbk,H^m_{\mathfrak{m}}(H^{n-1}_I(\omega_R)))$ 
		and $\dim_{\Bbbk} \Hom_R(\Bbbk,H^0_{\mathfrak{m}}(H^{m+n-2}_I(\omega_R)))$ are 
		not finite.
	\end{itemize} 
\end{example}

\begin{proof} 
	The ring homomorphism $\phi$ of \ref{con-1} (C) induces ring homomorphisms 
	\[
	\phi^{\alpha} : R/I^{\alpha} \to S/(\au)^{\alpha}S \; \mbox{ for all } \alpha \geq 1,
	\]
	which are well-defined.
	We claim that $\phi^{\alpha}$ is injective for all $\alpha \geq 1$. The 
	restriction $\psi^{\alpha}$ of the injection $\psi$ of \ref{con-1} (C) to degree $\alpha$ yields injective homomorphisms 
	$$
	{\psi}^{\alpha} : I^{\alpha}/I^{\alpha +1} \to (\au)^{\alpha}S /(\au)^{\alpha +1}S \; \mbox{ for all }\alpha \geq 0. 
	$$
	Let $D_{\alpha} := \Coker \psi^{\alpha}$ and let $\xx \cdot \au$ be the sequence of elements 
	$\{x_i a_j : i = 1,\ldots, n-1, j = 1,\ldots,m\}$. We define $C_{\alpha} 
	= \Coker \phi^{\alpha} = S/((\au)^{\alpha}S + S[[\xx\cdot\au]])$ 
	because of  $\im \phi = S[[\xx \cdot \au]]$. Now it follows that
	\[
	D_{\alpha} \cong ((\au)^{\alpha}S + S[[\xx \cdot \au]])/((\au)^{\alpha +1}S +S[[\xx \cdot \au]]).
	\]	
	as a consequence of the following commutative diagram with exact 
	rows and the snake lemma 
	\[
	\begin{array}{ccccccccc}
	0 & \to & I^{\alpha}/I^{\alpha +1} & \to & R/I^{\alpha +1}& \to &R/I^{\alpha} & \to & 0 \\
	  &     & \downarrow \psi^{\alpha} & & \downarrow \phi^{\alpha +1}& & \downarrow \phi^{\alpha}& & \\
	0 & \to & (\au)^{\alpha}S/(\au)^{\alpha +1}S & \to & S/(\au)^{\alpha +1}S & \to & S/(\au)^{\alpha}S & \to & 0. 
	\end{array}
	\]
	Since $\psi^{\alpha}$ is injective for all $\alpha \geq 0$ as the restriction of $\phi^{\alpha}$ it
	follows by induction that $\phi^{\alpha}$ is injective too. 
	Therefore $D_{\alpha}$ is spanned by $(\xx)^i (\au)^{\alpha}$ with $0 \leq i < \alpha$
	over $\Bbbk$. As an $R$-module it is finitely generated. For the radical 
	of the annihilator it follows that $\Rad_S \Ann_S D_{\alpha} = (\xx ,\au)$. 
	So that $C_{\alpha}, \alpha \geq 1,$ is an $R$-module of finite length.  
	There is the short exact sequence 
	\[
	0 \to R/I^{\alpha} \to S/(\au)^{\alpha}S \to C_{\alpha} \to 0 \; \mbox{ for all } \alpha \geq 1.
	\]
	Because of $C_{\alpha} \cong \oplus_{\beta = 1}^{\alpha-1} D_{\beta}$ it follows that 
	it is an $R$-module of finite length. Since $S/(\au)^{\alpha}S$ is a finitely generated 
	Cohen-Macaulay $R$-module it implies that $H^1_{\mathfrak{m}}(R/I^{\alpha}) \cong 
	C_{\alpha}$ and $H^i_{\mathfrak{m}}(R/I^{\alpha}) = 0$ for all $i \not= 1,m$. The 
	above short exact sequences form an inverse system of short exact sequences. 
	By passing to the inverse limit we get a short exact sequence 
	\[
	0 \to R \to S \to \varprojlim H^1_{\mathfrak{m}}(R/I^{\alpha}) \to 0.
	\]
	Note that the first family is given by surjective maps. By virtue of the Local Duality 
	Theorem for a Cohen-Macaulay ring (see \ref{not-1} (A)) there are isomorphisms 
	$$
	H^i_{\mathfrak{m}}(R/I^{\alpha}) \cong \Hom_R(\Ext_R^{m+n-1-i}(R/I^{\alpha},\omega_R),E)
	$$ 
	for all $\alpha \geq 1$ and by passing to the inverse limit $\varprojlim H^i_{\mathfrak{m}}(R/I^{\alpha}) \cong \Hom_R(H^{m+n-1-i}_I(\omega_R),E)$ for all $i$. 
	Therefore $H^i_I(\omega_R) = 0$ for $i \not= n-1,m+n-2$ and 
	$$
	\varprojlim H^1_{\mathfrak{m}}(R/I^{\alpha}) \cong \Hom_R(H^{m+n-2}_I(\omega_R),E) \not= 0.
	$$ 
	Because $S$ is not a finitely generated $R$-module $\Hom_R(H^{m+n-2}_I(\omega_R),E)$ is not finitely generated too. By the isomorphism
	\[
	\Bbbk \otimes_R \Hom_R(H^n_I(\omega_R),E) \cong \Hom_R(\Hom_R(\Bbbk, H^{m+n-2}_I(\omega_R)),E)
	\]
	and by Matlis Duality $\dim_{\Bbbk} \Hom_R(\Bbbk, H^{m+n-2}_I(\omega_R))$ 
	is not finite. 
	Moreover, $H^{m+n-2}_I(\omega_R) \cong H^{m+n-2}_I(R) \otimes_R\omega_R$ so that 
	$\Hom_R(H^{m+n-2}_I(\omega_R),E) \cong \Hom_R(\omega_R, \Hom_R(H^{m+n-2}_I(R),E))$. Because $\omega_R$ is a finitely generated $R$-module $\Hom_R(H^{m+n-2}_I(R),E)$ 
	is not finitely generated and -- as above -- $\dim_{\Bbbk} \Hom_R(\Bbbk, H^{m+n-2}_I(R))$ is not finite. 
	That is, (a) and (b) of \ref{ex-1} are proved.
	
	Now recall that $S/(\au)^{\alpha}S$ 
	is an $m$-dimensional Cohen-Macaulay ring, finitely generated over $R$ and $\dim_R C_{\alpha} = 0$. Therefore, 
	$$
	S/(\au)^{\alpha}S \cong \omega_R(\omega_R(R/I^{\alpha})) 
	\cong \Ext_R^{n-1}(\Ext_R^{n-1}(R/I^{\alpha}, \omega_R),\omega_R)
	$$
	(by view 
	of \ref{not-1} (D)). By passing to the inverse limit of the corresponding inverse and by  Theorem \ref{thm-1}  (a) it yields  that $S \cong \Hom_R(H^{n-1}_I(\omega_R),H^{n-1}_I(\omega_R))$, whence 
	(c) is shown.
	The first claim in (d) follows by \ref{thm-1} (b) since $S$ is not finitely generated 
	over $R$. The second one is clear by (b) since $\Supp_R H^{m+n-2}_I(R) = V(\mathfrak{m})$.	
\end{proof}

For the case of $m = n$ in \ref{ex-1} we get that $R$ is a Gorenstein ring and therefore 
$\omega_R \cong R$. For $m = n= 2$ we recover Hartshorne's example in this different context with additional properties. 
 
\begin{question} 
	Let $I \subset R$ denote an ideal of a local ring $(R,\mathfrak{m})$ with $c = \grade I.$ We do not know whether the endomorphism ring $\Hom_R(H^c_I(R),H^c_I(R))$ is in 
	general commutative and Noetherian. 
\end{question}

For further results about the endomorphism ring $\Hom_R(H^c_I(R),H^c_I(R))$ we refer also to \cite{Sp8}.

\bibliographystyle{siam}

\bibliography{hart-1}

\end{document}